\documentclass[mathpazo]{jpde}

\usepackage{amsmath,amssymb, amsfonts}
\usepackage{epic}
\usepackage{pgfcore}
\usepackage{graphicx}

\allowdisplaybreaks 

\usepackage[pagewise]{lineno}
\linenumbers
\nolinenumbers

\begin{document}
\hoffset -1.cm \voffset -1.1cm

\renewcommand\thisnumber{X}
\renewcommand\thisyear {XXXX}
\renewcommand\thismonth{X}
\renewcommand\stitle{Comparison Principles for 3-D Steady Potential Flow}
\renewcommand\thisvolume{XX}

\title{Comparison Principles for 3-D Steady Potential Flow in Spherical Coordinates}

\author[B. Long]
 {LONG Bingsong\corrauth}
\address{School of Mathematics and Statistics, Huanggang Normal University, Hubei 438000, China.
\\
\\{\rm Received 17 October 2024; Accepted 16 February 2025}}
 \emails{{\tt longbingsong@hgnu.edu.cn} (B. Long)}

\begin{abstract}
In this paper, we consider the 3-D steady potential flow for a compressible gas with pressure satisfying $p'(\rho)=\rho^{\gamma-1}$, where $\rho$ is the density and $\gamma\geq-1$ is a constant. In spherical coordinates, the potential equation is of mixed type in the unit sphere. We establish a strong comparison principle for elliptic solutions of the equation. The main difference from the classical case is that the coefficients of this equation depend fully on the potential function itself. We overcome this difficulty by the sufficient analysis on the structure of the equation itself, and finally derive the result. The result obtained here can be applied to the problem of supersonic flow over a delta wing and other problems related to gas dynamics.
\end{abstract}

 \clc{O175}

\maketitle

\section{Introduction}\label{sec: intro.}
We are concerned with the steady potential flow in the three-dimensional space. The Euler equations for potential flow take the form
\begin{align}
    &\mathrm{div}_{x}(\rho \nabla_{x}\Phi)=0,\label{eq: mass}\\
    &\frac{|\nabla_{x}\Phi|^{2}}{2}+h(\rho)=\frac{B}{2},\label{eq: Bernoulli law}
\end{align}
where $x=(x_1,x_2,x_3)$; $\rho$, $\Phi$ and $\frac{B}{2}$ stand for the density, velocity potential and Bernoulli constant, respectively. Here $h(\rho)$ is the specific enthalpy, satisfying
\begin{equation}\label{eq for h}
	h'(\rho)=\frac{p'(\rho)}{\rho}=\frac{c^2(\rho)}{\rho},
\end{equation}
where $p=p(\rho)$ is the pressure and $c=c(\rho)$ is the sound speed, defined by
\begin{equation}\label{eq:state eq.}
    p'(\rho)=\rho^{\gamma-1}, \qquad~\text{for}~\gamma\geq -1.
\end{equation}
From (\ref{eq for h})-(\ref{eq:state eq.}), we have
\begin{equation*}
  h(\rho)=
      \begin{cases}
          \dfrac{\rho^{\gamma-1}-\rho^{\gamma-1}_0}{\gamma-1}=\dfrac{c^2-c^2_0}{\gamma-1},\quad&\gamma\in [-1,1)\cup(1,+\infty),\\
          \ln({\rho}/{\rho_0}),\quad&\gamma=1,
      \end{cases}
\end{equation*}
where $c_0$ is the sound speed at density $\rho_0>0$. Substituting (\ref{eq: Bernoulli law}) into (\ref{eq: mass}), we obtain a quasilinear equation of second order
\begin{align}\label{eq: 3-D eq.}
    &(c^{2}-\Phi^{2}_{x_{1}})\Phi_{x_{1}x_{1}}+(c^{2}-\Phi^{2}_{x_{2}})\Phi_{x_{2}x_{2}}+(c^{2}-\Phi^{2}_{x_{3}})\Phi_{x_{3}x_{3}}\nonumber\\
    &\qquad-2\Phi_{x_{1}}\Phi_{x_{2}}\Phi_{x_{1}x_{2}}-2\Phi_{x_{1}}\Phi_{x_{3}}\Phi_{x_{1}x_{3}}-2\Phi_{x_{2}}\Phi_{x_{3}}\Phi_{x_{2}x_{3}}=0,
\end{align}
where
\begin{equation}\label{eq:c}
     c^2=c^2_0+\frac{\gamma-1}{2}(B-|\nabla_{x}\Phi|^2).
\end{equation}

Euler equations not only has a broad background of applications \cite{Courant76}, but is also mathematically challenging to study. There has been no general mathematical theory for the high dimensional case even until now \cite{CFe11}. To study equation (\ref{eq: 3-D eq.}), scholars typically start with a simplified model of a particular physics problem, such as the problem of supersonic flow past a delta wing.

Since most supersonic aircraft are designed as a triangle \cite{Hayes04}, the problem of supersonic flow over a delta wing is important in aeronautics. A great deal of experimental and computational work has been done over the past few decades to study this problem (see \cite{Baba63,Fowe56,Hui71,MW84} and references therein). Whereas, to our best knowledge, little is known about the rigorous mathematical theory for the problem. A recent work is \cite{CY15}, in which the authors proved the global existence of the conic solution for the case that the sweep angle of the wing is almost close to zero. In practice, the sweep angle of supersonic aircraft is not that small, so further research is still needed. To study this case in the future, a very important step is to establish a strong comparison principle for elliptic solutions of equation (\ref{eq: phi eq.}) below.

Notice that equation (\ref{eq: 3-D eq.}) is invariant under the following scaling
\begin{equation*}
    x\longrightarrow \tau x, \quad 	(\rho,\Phi)\longrightarrow\Big(\rho,\dfrac{\Phi}{\tau}\Big), \qquad~\text{for}~\tau\neq 0.
\end{equation*}
Hence, we can find an important class of solutions to (\ref{eq: 3-D eq.}) in the spherical coordinates $(r,\theta, \varphi):=\left(\sqrt{x^2_1+x^2_2+x^2_3}, \arccos({x_1}/{r}), \arctan({x_3}/{x_2})\right)$ of the form
\begin{align}\label{eq: spherical tran.}
    \rho(x)=\rho(\zeta), \quad \Phi(x)=r\phi(\zeta),
\end{align}
where $\zeta=(\theta, \varphi)$.  This class of solutions arise in many physically relevant problems, such as the problem of supersonic flow past a conical object \cite{CY15,C03,LY19}. By (\ref{eq: 3-D eq.}) and (\ref{eq: spherical tran.}), the equation for $\phi$ has the form
\begin{align}\label{eq: phi zhankai}
    &(c^{2}-\partial^{2}_\theta\phi)\partial_{\theta\theta}\phi+\left(c^{2}-\left(\dfrac{\partial_\varphi\phi}{\sin\theta}\right)^2\right)\dfrac{\partial_{\varphi\varphi} \phi}{\sin^2\theta}-2\partial_\theta \phi\dfrac{ \partial_\varphi\phi}{\sin\theta}\dfrac{\partial_{\theta\varphi}\phi}{\sin\theta}\nonumber\\
    &\qquad+
    \cot\theta\left(c^{2}+\left(\dfrac{\partial_\varphi\phi}{\sin\theta}\right)^2\right)\partial_\theta\phi+\left(2c^2-\partial^{2}_\theta\phi-\left(\dfrac{\partial_\varphi\phi}{\sin\theta}\right)^2\right)\phi=0,
\end{align}
or equivalently,
\begin{align}\label{eq: phi eq.}
    \mathrm{div}_{\zeta}(\rho D_{\zeta}\phi)+2\rho\phi=0,
\end{align}
with
\begin{align}\label{eq:rho}    
\rho=\rho(|D_{\zeta}\phi|,\phi)=\left(c_{0}^{2}+\frac{\gamma-1}{2}\left(B-\phi^2-|D_{\zeta}\phi|^2\right)\right)^\frac{1}{\gamma-1},
\end{align}
where $D_{\zeta}=\left(\partial_\theta, \frac{\partial_\varphi}{\sin\theta}\right)$ denotes the gradient operator in the unit sphere; $\mathrm{div}_{\zeta}$~denotes the divergence operator in the unit sphere, satisfying
\begin{align}\label{eq:divzeta}
  \mathrm{div}_{\zeta}v=\dfrac{1}{\sin\theta}\partial_\theta(\sin\theta v_\theta)+\dfrac{\partial_\varphi v_\varphi}{\sin\theta},
\end{align}
with~$v=(v_\theta, v_\varphi)$.

Define a pseudo-Mach number
\begin{align}\label{eq: 2D Mach num.}
    L^{2}=L^{2}(|D_{\zeta}\phi|,\phi):=\frac{|D_{\zeta}\phi|^{2}}{c^2}.
\end{align}
 The type of equation (\ref{eq: phi eq.}) is determined by $L$, with $L>1$ for hyperbolic regions, $0\leq L<1$ for elliptic regions, and $L=1$ for parabolic regions. We point here that the following ellipticity principle was established in \cite{LY21}: for $\gamma\geq -1$, equation (\ref{eq: phi eq.}) is elliptic in the interior of a parabolic-elliptic region; in particular, when $\gamma>-1$, there exists a domain-dependent function $\delta_0>0$ such that $L\leq 1-\delta_0$. Moreover, there are no open parabolic regions.

It is known that a systematic mathematical theory of comparison principles of linear elliptic equations has been established (see \cite[Chapters 4, 8]{GT03}). Gilbarg-Trudinger also established several classical comparison principles for elliptic solutions of quasilinear equations (see \cite[Chapter 10]{GT03} for details). These results have been widely used in the study of elliptic equations. However, there are also nonlinear equations to which the conclusions in \cite{GT03} cannot be applied directly, such as two-dimensional unsteady potential flow equations and three-dimensional steady potential flow equations. For the two-dimensional unsteady case, Chen-Feldman \cite{CFe12} established a strong comparison principle for elliptic solutions of the self-similar potential flow equation. This comparison principle plays an important role in the analysis of the shock reflection problem \cite{CFe18}. However, the results of the three-dimensional steady case are not yet available. Then, in this paper, we mainly focus on comparison principles for (\ref{eq: phi eq.}), motivated by the application to the problem of supersonic flow over a delta wing, and the work of comparison principles for self-similar potential flow \cite{CFe12}. 

We remark here that the state of equation of the Chaplygin gas satisfies the relation (\ref{eq:state eq.}) with $\gamma=-1$. This gas has been used extensively in cosmology \cite{KMP01,Popov10}. Particularly, based on the properties of the Chaplygin gas, some multidimensional Riemann problems have been well studied (see, for instance, \cite{CQu12,Wang20,SL16,Serre09}).

We denote by $\mathcal{N}$ the differential operator on the left-hand side of (\ref{eq: phi eq.}):
\begin{align}\label{for N}
    \mathcal{N}\phi:= \mathrm{div}_{\zeta}(\rho D_{\zeta}\phi)+2\rho\phi.
\end{align}
The following theorem is the main result of this paper.
\begin{theorem}[Main Theorem]\label{thm: Main} \it
	Let $\Omega$ be an open bounded set in the unit sphere. Let $p$ satisfy (\ref{eq:state eq.}) with $\gamma\geq-1$ on $\{\rho>0\}$. Also, let ${\phi}_{\pm}\in C^{0,1}(\overline{\Omega})\cap C^2(\Omega)$ satisfy
    \begin{align*}
	\mathcal{N} \phi_-\geq0, \quad\mathcal{N} \phi_+\leq0, \quad \text{in}~\Omega.
	\end{align*}
Assume that $\phi_{\pm}\geq c$ and
	\begin{align}
	\rho(|D_{\zeta}\phi_-|,\phi_-)>0, \quad L^{2}(|D_{\zeta}\phi_-|,\phi_-)<1,\label{elliptic condition1}\\   \rho(|D_{\zeta}\phi_+|,\phi_+)>0,\quad L^{2}(|D_{\zeta}\phi_+|,\phi_+)<1,\label{elliptic condition2}
	\end{align}
in $\Omega$. Then, if $\phi_-\leq\phi_+$ on the boundary $\partial\Omega$, it follows that $\phi_-\leq\phi_+$ in $\Omega$; moreover, either $\phi_-<\phi_+$ in $\Omega$ or $\phi_-\equiv\phi_+$ in $\Omega$.
\end{theorem}

To obtain Theorem \ref{thm: Main}, we also establish a weak comparison principle (Theorem \ref{weak comparison principle} below) and a Hopf-type lemma for equation (\ref{eq: phi eq.}) (Lemma \ref{Hopf lemma for N} below).

 The rest of this paper is arranged as follows. In Section \ref{sec: weak comparison principle}, we first give a comparison principle for nonuniformly elliptic equation in divergence form (Lemma \ref{comparison principle for Q} below), and then apply it to establish a weak comparison principle for equation (\ref{eq: phi eq.}). In Section \ref{sec: Main part}, we establish a Hopf-type lemma for equation (\ref{eq: phi eq.}) and use it to complete the proof of Theorem \ref{thm: Main}.

\section{A weak comparison principle for equation (\ref{eq: phi eq.})}\label{sec: weak comparison principle}
Motivated by the work of \cite{CFe12}, we first construct a comparison principle for nonuniformly elliptic equation in the divergence form.

Let $\Omega$ be an open set in the unit sphere. We consider the nonlinear equation
\begin{align}\label{eq: for Q}
    \mathcal{Q}\phi=\mathrm{div}_{\zeta}(A(D_{\zeta}\phi,\phi,\zeta))+B(D_{\zeta}\phi,\phi,\zeta)=0,
\end{align}
where the functions $A(q,z,\zeta)=\big(A_1(q,z,\zeta), A_2(q,z,\zeta)\big)$ and $B(q,z,\zeta)$ are defined on a subset of $\mathbb{R}^2\times\mathbb{R}\times\Omega$.

Let $\phi\in C^1(\Omega)$. For $D\subset\Omega$, denote a set
\begin{equation}\label{eq: for Sigma}
    \Sigma(\phi,D):=\{(q,z,\zeta):q=D_{\zeta}\phi(\zeta),z=\phi(\zeta),\zeta\in D\}\subset \mathbb{R}^2\times\mathbb{R}\times\Omega.
\end{equation}
From the form of the operator $\mathrm{div}_{\zeta}$ defined by (\ref{eq:divzeta}), we can give the following definition by a  trivial calculation.

\begin{definition}\label{def weak solution} \it
    Let $\Omega$ be an open set in the unit sphere. Let ${\phi}\in C^{0,1}(\overline{\Omega})\cap C^2(\Omega)$.
    Then, $\mathcal{Q}\phi\geq 0$ (resp. $\leq 0$) in the weak sense if the functions $A(q,z,\zeta)$ and $B(q,z,\zeta)$ are bounded on $\Sigma(\phi,\Omega)$ and continuous in a neighborhood of $\Sigma(\phi,D)$ for any compact set $D\subset\Omega$, and
    \begin{equation}\label{weak sense} \int_{\Omega}\left(A(D_{\zeta}\phi,\phi,\zeta)\cdot D_{\zeta} \psi-B(D_{\zeta}\phi,\phi,\zeta)\psi\right) {\rm d}\zeta\leq 0 \quad (resp. \geq 0)
    \end{equation}
    for all nonnegative functions $\psi\in C^{1}_{c}(\Omega)$, where ${\rm d}\zeta=\sin\theta {\rm d}\theta {\rm d}\varphi$.
\end{definition}

For $\phi_{\pm}\in C^{0,1}(\overline{\Omega})\cap C^1(\Omega)$, denote
\begin{equation}\label{eq: for Sigma2}
    \Sigma([\phi_{-},\phi_{+}],\Omega):=\{(q,z,\zeta):q=D_{\zeta}\phi_t(\zeta),z=\phi_t(\zeta),\zeta\in \Omega, t\in[0,1]\}\subset \mathbb{R}^2\times\mathbb{R}\times\Omega,
\end{equation}
with $\phi_t(\zeta)=t\phi_-(\zeta)+(1-t)\phi_+(\zeta)$.

\begin{lemma}\label{comparison principle for Q} \it
 Let $\Omega$ be an open bounded set in the unit sphere. Let ${\phi}_{\pm}\in C^{0,1}(\overline{\Omega})\cap C^1(\Omega)$ satisfy
    \begin{align*}
	\mathcal{Q} \phi_-\geq0, \quad\mathcal{Q} \phi_+\leq0,\quad in~\Omega
	\end{align*}
 in the weak sense. Moreover, assume that the functions $A(q,z,\zeta)$ and $B(q,z,\zeta)$ satisfy the following:

 $(i)$ $A(q,z,\zeta)$ and $B(q,z,\zeta)$ are bounded on $\Sigma([\phi_{-},\phi_{+}],\Omega)$, and $D_{q,z}(A,B)$ exist and are continuous in a neighborhood of $\Sigma([\phi_{-},\phi_{+}],D)$ for any compact set $D\subset\Omega$.

 $(ii)$ There exists a constant $\beta\in(0,1]$ such that the $3\times3$ matrix
\begin{equation}\label{matrix H}
H=[H_{ij}]:={
\left[ \begin{array}{ccc}
\partial_{q_1}A_1(q,z,\zeta)& \partial_{q_1}A_2(q,z,\zeta)& -\beta\partial_{q_1}B(q,z,\zeta)\\
\partial_{q_2}A_1(q,z,\zeta)& \partial_{q_2}A_2(q,z,\zeta)& -\beta\partial_{q_2}B(q,z,\zeta)\\
\partial_{z}A_1(q,z,\zeta)& \partial_{z}A_2(q,z,\zeta)& -\beta\partial_{z}B(q,z,\zeta)
\end{array}
\right ]}
\end{equation}
is nonnegative; that is,
\begin{equation}\label{for nonnegative}
\sum^{3}_{i,j=1}H_{ij}(q,z,\zeta)\xi_i\xi_j\geq 0
\end{equation}
for any $(q,z,\zeta)\in\Sigma([\phi_{-},\phi_{+}],\Omega)$ and $\xi=(\xi_1,\xi_2,\xi_3)\in \mathbb{R}^3$.

 $(iii)$ For each $(q,z,\zeta)\in\Sigma([\phi_{-},\phi_{+}],\Omega)$,
 \begin{equation}\label{for postive}
\sum^{3}_{i,j=1}H_{ij}(q,z,\zeta)\xi_i\xi_j> 0
\end{equation}
for all $\xi=(\xi_1,\xi_2,\xi_3)\in \mathbb{R}^3$ with $(\xi_1,\xi_2)\neq 0$.

Then, if $\phi_-\leq \phi_+$ on $\partial \Omega$, it follows that $\phi_-\leq \phi_+$ in $\Omega$.
\end{lemma}

\begin{proof}
The basic idea of the proof of this lemma is referred to \cite[Proposition 2.1]{CFe12}, but the expression of the equation requires some modification. We still write the detailed proof here for the convenience of our readers.

It is obvious that (\ref{weak sense}) also holds for all nonnegative functions $\psi\in H^{1}_{0}(\Omega)$ by approximation. Under the assumptions of this lemma, the following function satisfies
\begin{equation}\label{w+}
h^+:=\max(\phi_- -\phi_+,0)\in C^{0,1}(\overline{\Omega})\cap H^{1}_{0}(\Omega)\cap C^1(\{h^+>0\}\cap\Omega).
\end{equation}
Let $\psi=(h^+)^{\frac{1}{\beta}}$. It follows from $\beta\in(0,1]$ that $\psi\in H^{1}_{0}(\Omega)$. Inserting this function $\psi$ into (\ref{weak sense}) for $\phi_{\pm}$, we have
	\begin{align}
	\int_{\Omega}\left(A(D_{\zeta}\phi_-,\phi_-,\zeta)\cdot D_{\zeta}(h^+)^{\frac{1}{\beta}}-B(D_{\zeta}\phi_-,\phi_-,\zeta)(h^+)^{\frac{1}{\beta}}\right) {\rm d}\zeta\leq 0,\label{for phifu}\\   
    \int_{\Omega}\left(A(D_{\zeta}\phi_+,\phi_+,\zeta)\cdot D_{\zeta}(h^+)^{\frac{1}{\beta}}-B(D_{\zeta}\phi_+,\phi_+,\zeta)(h^+)^{\frac{1}{\beta}}\right) {\rm d}\zeta\geq 0.\label{for phizheng}
	\end{align}
Then, from (\ref{for phifu})-(\ref{for phizheng}), we obtain
\begin{align}\label{weak sense h}
&\int_{\Omega}\Big(\big(A(D_{\zeta}\phi_-,\phi_-,\zeta)-A(D_{\zeta}\phi_+,\phi_+,\zeta)\big)\cdot D_{\zeta}(h^+)^{\frac{1}{\beta}}\nonumber\\
&\qquad-
\big(B(D_{\zeta}\phi_-,\phi_-,\zeta)-B(D_{\zeta}\phi_+,\phi_+,\zeta)\big)(h^+)^{\frac{1}{\beta}} \Big){\rm d}\zeta\leq 0.
\end{align}

Let us define $F(\zeta)$ the integrand on the left-hand side of (\ref{weak sense h}), and
       \begin{align} 
          a_{ij}(\zeta)=\int^{1}_{0}\partial_{q_j}A_i(D_{\zeta}\phi_t,\phi_t,\zeta){\rm d}t,\quad &b_{i}(\zeta)=\int^{1}_{0}\partial_{z}A_i(D_{\zeta}\phi_t,\phi_t,\zeta){\rm d}t,\nonumber\\
  	c_{i}(\zeta)=\int^{1}_{0}\partial_{q_i}B(D_{\zeta}\phi_t,\phi_t,\zeta){\rm d}t,\quad   &d(\zeta)=\int^{1}_{0}\partial_{z}B(D_{\zeta}\phi_t,\phi_t,\zeta){\rm d}t,\label{cof aij bi}
       \end{align}  
where $\phi_t=t\phi_-+(1-t)\phi_+$. It is easy to verify that $F(\zeta)\in L^{\infty}(\overline{\Omega})\cap C(\Omega)$ and $a_{ij},b_i,c_i,d\in C(\Omega)$. Moreover, using the assumption $(i)$, we have
\begin{equation}\label{for F}
F(\zeta)=\frac{1}{\beta}(h^+)^{\frac{1}{\beta}-1}\sum^{2}_{i,j=1}\left(a_{ij}\partial_i h^+\partial_j h^++b_{i} h^+\partial_i h^+-\beta c_{i} h^+\partial_i h^+-\beta d (h^+)^2\right),\quad\text{in} ~\Omega,
\end{equation}
where $(\partial_1,\partial_2)=(\partial_\theta,\frac{\partial_\varphi}{\sin\theta})$.
From the assumption $(ii)$, with the choice $(\xi_1,\xi_2,\xi_3)=(\partial_\theta h^+, \frac{\partial_\varphi h^+}{\sin\theta}, h^+)$, it follows that $F(\zeta)\geq 0$ in $\Omega$. Besides, (\ref{weak sense h}) implies that $F(\zeta)\leq 0$ in $\Omega$. Therefore, we obtain $F(\zeta) \equiv 0$ in $\Omega$. Then, on $\{h^+>0\}\cap\Omega$, there is
\begin{equation}\label{for h=0}
\sum^{2}_{i,j=1}\left(a_{ij}\partial_i h^+\partial_j h^++b_{i} h^+\partial_i h^+-\beta c_{i} h^+\partial_i h^+-\beta d (h^+)^2\right)\equiv 0.
\end{equation}
From (\ref{for h=0}) and the assumption $(iii)$, it follows that
\begin{equation}\label{for Dh=0}
D_{\zeta}h^+\equiv 0, \quad\text{on} ~\{h^+>0\}\cap\Omega.
\end{equation}
On the other hand, the function $h^+\in C^{0,1}(\overline{\Omega})$ and $h^+\geq 0$ in $\Omega$ with $h^+\equiv 0$ on $\partial\Omega$. In conclusion, we have $h^+\equiv 0$ in $\Omega$; that is, $\phi_-\leq\phi_+$ in $\Omega$.
\end{proof}

\begin{remark}
 From the assumption $(iii)$, with the choice $(\xi_1,\xi_2)\neq 0$ and $\xi_3=0$, it follows from condition (\ref{for postive}) that the $2\times2$ matrix 
 \begin{equation}
{
\left[ \begin{array}{ccc}
\partial_{q_1}A_1(q,z,\zeta)& \partial_{q_1}A_2(q,z,\zeta)\\
\partial_{q_2}A_1(q,z,\zeta)& \partial_{q_2}A_2(q,z,\zeta)
\end{array}
\right ]}
\end{equation}
is positive for all $(q,z,\zeta)\in\Sigma(\phi_t,\Omega)$ with any $t\in[0,1]$. This implies that the nonlinear operator $\mathcal{Q} \phi$ is elliptic in $\Omega$ for $\phi_t$ for any $t\in[0,1]$.
 \end{remark}

Next, we establish a weak comparison principle for equation (\ref{eq: phi eq.}), by applying Lemma \ref{comparison principle for Q}. Note that equation (\ref{eq: phi eq.}) has the form (\ref{eq: for Q}) with
\begin{equation}\label{for A and B}
A(D_{\zeta}\phi,\phi,\zeta)=\rho(|D_{\zeta}\phi|,\phi)D_{\zeta}\phi,\quad B(D_{\zeta}\phi,\phi,\zeta)=2\rho(|D_{\zeta}\phi|,\phi)\phi,
\end{equation}
where the expression of $\rho(|D_{\zeta}\phi|,\phi)$ is given by (\ref{eq:rho}).

\begin{theorem}\label{weak comparison principle} \it
Let $\Omega$ be an open bounded set in the unit sphere. Let $p$ satisfy (\ref{eq:state eq.}) with $\gamma\geq-1$ on $\{\rho>0\}$. Also, let ${\phi}_{\pm}\in C^{0,1}(\overline{\Omega})\cap C^2(\Omega)$ satisfy
    \begin{align}\label{for N weak}
	\mathcal{N} \phi_-\geq0, \quad\mathcal{N} \phi_+\leq0,\quad in~\Omega,
    \end{align}
where $\mathcal{N}$ is defined by (\ref{for N}). Assume that $\phi_{\pm}\geq c$ and
	\begin{align*}
	\rho(|D_{\zeta}\phi_-|,\phi_-)>0, \quad L^{2}(|D_{\zeta}\phi_-|,\phi_-)<1,\\
 \rho(|D_{\zeta}\phi_+|,\phi_+)>0,\quad L^{2}(|D_{\zeta}\phi_+|,\phi_+)<1,
	\end{align*}
in $\Omega$. Then, if $\phi_-\leq\phi_+$ on the boundary $\partial\Omega$, it follows that $\phi_-\leq\phi_+$ in $\Omega$.
\end{theorem}

\begin{proof}
 The proof will be divided into two steps.

 \emph{Step} 1: Let $\phi_t=t\phi_-+(1-t)\phi_+$ for $t\in[0,1]$. Under the condition of this lemma, we assert here that $L^{2}(|q|,z):=L^{2}(|D_{\zeta}\phi_t|,\phi_t)<1$ in $\Omega$ for all $t\in[0,1]$. To this end, we only need to show that $(q,z)\rightarrow |q|^2-c^2(q,z)$ is a convex function.

 We first compute the second derivatives of $c^2(|q|,z)$. From (\ref{eq for h})-(\ref{eq:state eq.}) and (\ref{eq:rho}), it follows that
 \begin{align}
 c^2=c^2(|q|,z)=c_{0}^{2}+\frac{\gamma-1}{2}(B-z^2-|q|^2).
\end{align}
 Then, using the notation $\partial_{q_{3}}:=\partial_{z}$, we know that, for $i,j=1,2,3$
\begin{equation}\label{for d c^2}
\partial_{q_iq_j}c^2(q,z)=-(\gamma-1)\delta_{ij},
\end{equation}
where $\delta_{ij}$ equals $1$ if $i=j$, and equals $0$ otherwise.

With the use of (\ref{for d c^2}), a trivial calculation yields that for any $\xi=(\xi_1,\xi_2,\xi_3)\in \mathbb{R}^3$,
\begin{equation}\label{for d q^2-c^2}
\sum^3_{i,j=1}\partial_{q_iq_j}\left(|q|^2-c^2(q,z)\right)\xi_i\xi_j=(\gamma+1)|\xi|^2.
\end{equation}
Owing to $\gamma\geq-1$, the left-hand side of (\ref{for d q^2-c^2}) is nonnegative.  Hence, our assertion have proved.

\emph{Step} 2:  Under the assertion proved in \emph{Step} 1, we can check that the conditions of Lemma \ref{comparison principle for Q} are all satisfied here.

Since $\rho(|D_{\zeta}\phi_-|,\phi_-)>0$, $\rho(|D_{\zeta}\phi_+|,\phi_+)>0$ and ${\phi}_{\pm}\in C^{0,1}(\overline{\Omega})\cap C^2(\Omega)$, equation (\ref{for N weak}) also holds in the weak sense of Definition \ref{def weak solution}. Moreover, combining $\rho(|D_{\zeta}\phi_-|,\phi_-)>0$, $\rho(|D_{\zeta}\phi_+|,\phi_+)>0$ and (\ref{eq:rho}), it is obvious that $\rho(|D_{\zeta}\phi_t|,\phi_t)>0$. Then, using the regularity of ${\phi}_{\pm}$ again and (\ref{for A and B}), we know that the condition $(i)$ of Lemma \ref{comparison principle for Q} is satisfied.

By a direct computation, we find that matrix (\ref{matrix H}), for equation (\ref{eq: phi eq.}) with $\beta=\dfrac{1}{2}$ is the following $3\times 3$ matrix:
\begin{align}\label{matrix Hqz1}
H=[H_{ij}(q,z,\zeta)]:=&{
\left[ \begin{array}{ccc}
\rho-q^2_1/\rho^{\gamma-2} & -q_1q_2/\rho^{\gamma-2}& q_1z/\rho^{\gamma-2}\\
-q_1q_2/\rho^{\gamma-2} & \rho-q^2_2/\rho^{\gamma-2}&  q_2z/\rho^{\gamma-2}\\
-q_1z/\rho^{\gamma-2} & -q_2z/\rho^{\gamma-2}&  z^2/\rho^{\gamma-2}-\rho
\end{array}
\right ]},
\end{align}
 where $q=(q_1,q_2)=\left(\partial_\theta\phi_t,\frac{\partial_\varphi\phi_t}{\sin\theta}\right)$, $z=\phi_t$. With the use of the relation $c^2=\rho^{\gamma-1}$, (\ref{matrix Hqz1}) can be written as
\begin{align}\label{matrix Hqz2}
H=[H_{ij}(q,z,\zeta)]:=
\frac{1}{\rho^{\gamma-2}}{
\left[ \begin{array}{ccc}
c^2-q^2_1 & -q_1q_2& q_1z\\
-q_1q_2 & c^2-q^2_2&  q_2z\\
-q_1z & -q_2z&  z^2-c^2
\end{array}
\right ]}.
\end{align}
Then, we have
\begin{align}
           \sum^{3}_{i,j=1}H_{ij}(q,z,\zeta)\xi_i\xi_j&=\frac{1}{\rho^{\gamma-2}}\left((c^2-q^2_1)\xi^2_1+(c^2-q^2_2)\xi^2_2-2q_1q_2\xi_1\xi_2+(z^2-c^2)\xi^2_3\right)\nonumber\\
&=\frac{1}{\rho^{\gamma-2}}\left(c^2(\xi^2_1+\xi^2_2)-(q_1\xi_1+q_2\xi_2)^2+(z^2-c^2)\xi^2_3\right), \label{for Hij}
\end{align}
for all $\xi=(\xi_1,\xi_2,\xi_3)\in \mathbb{R}^3$.

By the Schwartz inequality $(q_1\xi_1+q_2\xi_2)^2\leq (q^2_1+q^2_2)(\xi^2_1+\xi^2_1)$, we conclude from (\ref{for Hij}) that for all $\xi=(\xi_1,\xi_2,\xi_3)\in \mathbb{R}^3$,
    \begin{align}\label{for Hij phit}
           \sum^{3}_{i,j=1}H_{ij}(q,z,\zeta)\xi_i\xi_j&=\frac{1}{\rho^{\gamma-2}}\left(c^2(\xi^2_1+\xi^2_2)-(q_1\xi_1+q_2\xi_2)^2+(z^2-c^2)\xi^2_3\right)\nonumber\\
           &\geq \frac{1}{\rho^{\gamma-2}}\left(c^2(\xi^2_1+\xi^2_2)-(q^2_1+q^2_2)(\xi^2_1+\xi^2_1)+(z^2-c^2)\xi^2_3\right)\nonumber\\
         &=\frac{1}{\rho^{\gamma-2}}\left((c^2-|q|^2)(\xi^2_1+\xi^2_1)+(z^2-c^2)\xi^2_3\right),
    \end{align}
where $|q|^2=q^2_1+q^2_2$. Due to $\phi_{\pm}\geq c$, there is $\phi_t=t\phi_-+(1-t)\phi_+\geq c$ for all $t\in[0,1]$. By this inequality, $\rho>0$ and the the assertion proved in \emph{Step} 1, the last expression in (\ref{for Hij phit}) is nonnegative for all $\xi=(\xi_1,\xi_2,\xi_3)\in \mathbb{R}^3$ and strictly positive when $(\xi_1,\xi_2)\neq 0$. Thus, matrix (\ref{matrix Hqz2}) satisfies the conditions $(ii)$--$(iii)$ of Lemma 
\ref{comparison principle for Q}.

In conclusion, we complete the proof of this theorem by Lemma \ref{comparison principle for Q}.
\end{proof}

 \begin{remark}\label{weak condition}
 By (\ref{eq: spherical tran.}), a trivial calculation yields $\phi=u_r$, where $u_r$ is the component of velocity along the radial direction. Then, we can assume without loss of generality that $\phi\geq 0$. Combining this assumption and (\ref{eq: phi eq.}), we have $\mathrm{div}_{\zeta}(\rho D_{\zeta}\phi)\leq 0$. This implies that for elliptic solutions of (\ref{eq: phi eq.}), $\phi$ cannot achieve the local minimum anywhere in $\Omega$ by the maximum principle. Hence, to ensure $\phi\geq c$ in $\Omega$, we only need to assume that $\phi\geq c_*$ on $\partial\Omega$, where $c_*$ is the limit sonic. In the application to the problem of supersonic flow past a delta wing, to utilizing Theorem \ref{weak comparison principle} (or Theorem \ref{thm: Main}), the velocity of oncoming flow along the radial direction should satisfy $u_r\geq c_*$ on the shock surface because the potential function $\phi=u_r$ is continuous across the shock.
 \end{remark}

\section{The proof of Theorem \ref{thm: Main}}\label{sec: Main part}
  In order to prove Theorem \ref{thm: Main}, we first show the following lemma.
 \begin{lemma}\label{lemma1} \it
 	Let $\Omega$ be an open bounded set in the unit sphere. Let $\rho>0$ in $\Omega$. Then equation (\ref{eq: phi zhankai}) is uniformly elliptic in $\Omega$ if and only if there exists a positive number $\varepsilon\in(0,1)$ such that $L^2\leq 1-\varepsilon$ in $\Omega$, where $L$ is given by (\ref{eq: 2D Mach num.}).
 \end{lemma}

 \begin{proof}
 	A direct computation yields that the characteristic equation for (\ref{eq: phi zhankai}) has the form
 	\begin{equation*}
 		\lambda^2-a_1\lambda+a_2=0,
 	\end{equation*}
 	where
 	\begin{equation*}
 		\begin{aligned}
 			a_1&=2c^2-|D_{\zeta}\phi|^2,\\
 			a_2&=c^2(c^2-|D_{\zeta}\phi|^2).
 		\end{aligned}
 	\end{equation*}
 	Thus the ratio of the eigenvalues is
 	\begin{equation}\label{eq2}
 		\dfrac{\lambda_2}{\lambda_1}=\dfrac{a_1+\sqrt{a^2_1-4a_2}}{a_1-\sqrt{a^2_1-4a_2}}=\dfrac{c^2}{c^2-|D_{\zeta}\phi|^2}=\dfrac{1}{1-L^2}.
 	\end{equation}
       This equality implies that for any bounded constant $C>0$, if the ratio $\lambda_2/\lambda_1\leq C$, then we have
 	\begin{equation}\label{eq6}
 		L^2\leq 1-\frac{1}{C}.
 	\end{equation}
 	
 	On the other hand, from (\ref{eq: phi zhankai}), we obtain that for any $\tau\in \mathbb{R}^{2}-\{0\}$,
 	\begin{equation}\label{eq7}
 		\sum^{2}_{i,j=1} A_{ij}( \phi,D_{\zeta}\phi)\tau_i\tau_j=c^2|\tau|^2-(D_{\zeta}\phi\cdot\tau)^2,
 	\end{equation}
  where $A_{ij}(\phi,D_{\zeta}\phi)$ has its expression as shown in the second-order coefficients of (\ref{eq: phi zhankai}). By (\ref{eq7}) and the Cauchy-Schwartz inequality, there is
 	\begin{equation}\label{eq8}
 		c^2|\tau|^2 \geq \sum^{2}_{i,j=1} A_{ij}(\phi,D_{\zeta}\phi)\tau_i\tau_j\geq (c^2-|D_{\zeta}\phi|^2)|\tau|^2.
 	\end{equation}
  From (\ref{eq7})-(\ref{eq8}), it follows that
 	\begin{equation*}
 		\frac{\lambda_2}{\lambda_1}\leq\frac{1}{1-L^2}.
 	\end{equation*}
 	This means that if $L^2\leq 1-\varepsilon$, then
 	\begin{equation*}
 	   \frac{\lambda_2}{\lambda_1}\leq \frac{1}{\varepsilon},
 	\end{equation*}
 	which completes the proof.
 \end{proof}
Next, inspired by the work in \cite{CFe12,GT03}, we establish a Hopf-type lemma for equation (\ref{eq: phi eq.}).
\begin{lemma}\label{Hopf lemma for N} \it
Let $\Omega$ be an open bounded set in the unit sphere. Let $p$ satisfy (\ref{eq:state eq.}) with $\gamma\geq-1$ on $\{\rho>0\}$. Also, let ${\phi}_{\pm}\in C^{0,1}(\overline{\Omega})\cap C^2(\Omega)$ satisfy
    \begin{align*}
	\mathcal{N} \phi_-\geq0, \quad\mathcal{N} \phi_+\leq 0,\quad in~\Omega,
	\end{align*}
where $\mathcal{N}$ is defined by (\ref{for N}). Assume that $\phi_{\pm}\geq c$ and (\ref{elliptic condition1})-(\ref{elliptic condition2}) holds. Let $\phi_{\pm}$ and $\zeta_*\in\partial\Omega$ satisfy
	\begin{align}
	&\phi_-<\phi_+~\text{in}~\Omega,\qquad \phi_-(\zeta_*)=\phi_+(\zeta_*);\label{hopf condition one}\\
        &\phi_{\pm}\in C^2\big(\overline{\Omega}\cap C_R(\zeta_*)\big)\quad \text{for~some}~R>0;\label{hopf condition two}\\
        &\text{strict inequalities in (\ref{elliptic condition1})-(\ref{elliptic condition2}) hold at}~ \zeta_*;\label{hopf condition three}\\
        &\partial\Omega~\text{satisfies an interior sphere condition at}~\zeta_*.\label{hopf condition four}
	\end{align}
Then $\partial_{\nu}(\phi_--\phi_+)(\zeta_*)>0$, where $\nu$ is the unit outer normal derivative of $\phi_--\phi_+$ at $\zeta_*$.
\end{lemma}

 \begin{proof}
 It follows from (\ref{hopf condition four}) that there exists a ball $C_r(\zeta_0)\subset \Omega$ with center $\zeta_0\in \Omega$ and radius $r>0$, such that $\zeta_*\in \partial C_r(\zeta_0)$. Then, using (\ref{hopf condition two})-(\ref{hopf condition three}) and choosing a suitable $r$, we can assume that there exists $\varepsilon>0$ such that
    \begin{align}
    \rho(|D_{\zeta}\phi_-|,\phi_-)\geq \varepsilon, \quad  L^{2}(|D_{\zeta}\phi_-|,\phi_-)\leq 1-\varepsilon\quad&\text{in}~ \overline{C_r(\zeta_0)},\label{uniformly elliptic condition one}\\ \rho(|D_{\zeta}\phi_+|,\phi_+)\geq \varepsilon, \quad  L^{2}(|D_{\zeta}\phi_+|,\phi_+)\leq 1-\varepsilon\quad&\text{in}~\overline{C_r(\zeta_0)},\label{uniformly elliptic condition two}
    \end{align}
 and
 \begin{equation}\label{for Cr}
  C_r(\zeta_0)\subset \Omega\cap C_R(\zeta_*).
 \end{equation}

 Let $h=\phi_--\phi_+$ in $C_r(\zeta_0)$. Then, by (\ref{hopf condition one}), we have $h<0$ in $C_r(\zeta_0)$ and $h=0$ at the point $\zeta_*$. Moreover, from (\ref{hopf condition two}) and (\ref{for Cr}), the function $h\in C^2(\overline{C_r(\zeta_0)})$ satisfies the following equation
 \begin{equation}\label{for w}
 \frac{1}{\sin\theta}\partial_i\left(\sin\theta a_{ij}(\zeta)\partial_jh+\sin\theta b_i(\zeta)h\right)+c_i(\zeta)\partial_ih+d(\zeta)h\geq 0,\quad i,j=1,2,
 \end{equation}
 where $(\partial_1,\partial_2)=\left(\partial_\theta,\frac{\partial_\varphi}{\sin\theta}\right)$; the coefficients $a_{ij},b_i,c_i,d$ are defined by (\ref{cof aij bi}) with the functions $A$ and $B$ given by (\ref{for A and B}). Moreover, proceeding as in the analysis of \emph{Step} 1 in Theorem \ref{weak comparison principle}, it is obvious that the relations in (\ref{uniformly elliptic condition one})-(\ref{uniformly elliptic condition two}) also hold for each $\phi_t=t\phi_-+(1-t)\phi_+$ with $t\in[0,1]$ in $\overline{C_r(\zeta_0)}$. It follows from Lemma \ref{lemma1} and (\ref{cof aij bi}) that equation (\ref{for w}) is uniformly elliptic in $\overline{C_r(\zeta_0)}$ and $a_{ij},b_i,c_i,d\in C^1(\overline{C_r(\zeta_0)})$.

 Also note that equation (\ref{for w}) has the form (\ref{eq: for Q}) with
\begin{align*}
 &A=(A_1,A_2), \quad \text{with} \quad A_i=\sum^2_{i=1} a_{ij}(\zeta)\partial_jh+b_i(\zeta)h,\\
 &B=\sum^2_{i=1} c_i(\zeta)\partial_ih+d(\zeta)h.
\end{align*}
For simplification, we also use $\mathcal{N}$ to denote the differential operator on the left-hand side of this equation. From $a_{ij},b_i,c_i,d\in C^1(\overline{C_r(\zeta_0)})$, it is clear that any ${\phi}_{\pm}\in C^{0,1}(\overline{C_r(\zeta_0)})\cap C^2(C_r(\zeta_0))$ with $\mathcal{N} \phi_-\geq0$ and $\mathcal{N} \phi_+\leq0$ in $C_r(\zeta_0)$ satisfy these equations in the weak sense of Definition \ref{def weak solution}. Moreover, it is direct to check that the conditions of Lemma \ref{comparison principle for Q} is satisfied in $C_r(\zeta_0)$ with $\beta=\frac{1}{2}$. Thus, equation (\ref{for w}) satisfies the weak comparison principle in the any open subregion $D\subset C_r(\zeta_0)$.

On the other hand, owing to $a_{ij},b_i,c_i,d\in C^1(\overline{C_r(\zeta_0)})$, equation (\ref{for w}) can be expanded as
 \begin{equation}\label{for h3.1}
 \hat{a}_{ij}(\zeta)\partial_{ij}h+ \hat{b}_i(\zeta)\partial_ih+\hat{c}(\zeta)h\geq 0,\quad i,j=1,2,
 \end{equation}
in $C_r(\zeta_0)$, where $\hat{b}_i,\hat{c}\in C(\overline{C_r(\zeta_0)})$ and $\hat{a}_{ij}=a_{ij}$. Note that equation (\ref{for h3.1}) is uniformly elliptic in $C_r(\zeta_0)$, and satisfies the weak comparison principle in any open subregion of $C_r(\zeta_0)$. Then, we conclude that Hopf's lemma holds for solutions of (\ref{for h3.1}) in $C_r(\zeta_0)$ by the standard proof of Hopf's lemma (see, for example,  in \cite[Lemma 3.4]{GT03}); that is, $\partial_{\nu}(h)(\zeta_*)>0$ for $h=\phi_--\phi_+$.
 \end{proof}

  Now, by the lemmas above, we are ready to prove Theorem \ref{thm: Main}. From Theorem \ref{weak comparison principle} and Lemma \ref{Hopf lemma for N}, we can follow the standard argument.

  Set $\Omega_1=\{\phi_-<\phi_+\}\cap\Omega$ and $\Omega_2=\{\phi_-=\phi_+\}\cap\Omega$. Obviously, $\Omega=\Omega_1\cup\Omega_2$ and $\Omega_1$ is a open region. If the conclusion in Theorem \ref{thm: Main} was false, then both $\Omega_1$ and $\Omega_2$ are nonempty. This means that there is a point $\zeta_*\in \partial\Omega_1\cap\Omega$, which has an interior touching ball from the $\Omega_1$ side, and
  \begin{equation}\label{for Dh1=0}
 D(\phi_--\phi_+)(\zeta_*)=0,
 \end{equation}
because the point $\zeta_*\in \partial\Omega_2\cap\Omega$. Since $\zeta_*$ is an interior point of $\Omega$, the condition of (\ref{hopf condition two})-(\ref{hopf condition three}) are satisfied by the result in \cite{LY21}. Then, an application of Lemma \ref{Hopf lemma for N} leads to a contradiction with (\ref{for Dh1=0}). Hence, we have proved Theorem \ref{thm: Main}.

\section*{Acknowledgement}
This work was supported in part by National Natural Science Foundation of Hubei Province of China (Grant number 2024AFB007) and Science and Technology Research Project of Education Department of Hubei Province (Grant number D20232901).

\end{document}